\documentclass[11pt,reqno]{amsart}
\usepackage{amssymb,latexsym}
\usepackage{cite}
\usepackage[height=190mm,width=130mm]{geometry}
\theoremstyle{plain}
\newtheorem{theorem}{Theorem}
\newtheorem{lemma}{Lemma}
\newtheorem{corollary}{Corollary}
\usepackage{enumerate}

\theoremstyle{definition}

\theoremstyle{remark}

\numberwithin{equation}{section}

\begin{document}
\title[Iterated Arithmetic Functions]{A Note about Iterated Arithmetic Functions}
\author{Colin Defant}
\address{University of Florida\\ Department of Mathematics\\ 1400 Stadium Rd.\\ 32611 Gainesville, FL\\ United States}

\email{cdefant@ufl.edu}

\begin{abstract}
Let $f\colon\mathbb{N}\rightarrow\mathbb{N}_0$ be a multiplicative arithmetic function such that for all primes $p$ and positive integers $\alpha$, $f(p^{\alpha})<p^{\alpha}$ and $f(p)\vert f(p^{\alpha})$. Suppose also that any prime that divides $f(p^{\alpha})$ also divides $pf(p)$. Define $f(0)=0$, and let $H(n)=\displaystyle{\lim_{m\rightarrow\infty}f^m(n)}$, where $f^m$ denotes the $m^{th}$ iterate of $f$. We prove that the function $H$ is completely multiplicative.   
\end{abstract}



\maketitle

\section{Introduction}
The study of iterated arithmetic functions, especially functions related to the Euler totient function $\varphi$, has burgeoned over the past century. In 1943, H. Shapiro's monumental work on a function $C(n)$, which counts the number of iterations of $\varphi$ needed for $n$ to reach $2$, paved the way for subsequent number-theoretic research \cite{shapiro43}. In this paper, we study a problem concerning the limiting behavior of iterations of functions related to the Euler totient function. 

Throughout this paper, we let $\mathbb{N}$, $\mathbb{N}_0$, and $\mathbb{P}$ denote the set of positive integers, the set of nonnegative integers, and the set of prime numbers, respectively. We will let $f\colon\mathbb{N}_0\rightarrow\mathbb{N}_0$ be a multiplicative arithmetic function which has the following properties for all primes $p$ and positive integers $\alpha$. 
\begin{enumerate}[I.]
\item $f(p^{\alpha})<p^{\alpha}$.
\item $f(p)\vert f(p^{\alpha})$. 
\item If $q$ is prime and $q\vert f(p^{\alpha})$, then $q\vert pf(p)$. 
\item $f(0)=0$. 
\end{enumerate} 
First, note that property IV does not effectively restrict the choice of $f$. Indeed, we may let $f$ be any multiplicative arithmetic function that satisfies properties I, II, and III and then simply define $f(0)=0$. One class of arithmetic functions which satisfy I, II, and III are the Schemmel totient functions. For each positive integer $r$, the Schemmel totient function $S_r$ is a multiplicative arithmetic function which satisfies 
\[S_r(p^{\alpha})=\begin{cases} 0, & \mbox{if } p\leq r; \\ p^{\alpha-1}(p-r), & \mbox{if } p>r \end{cases}\] 
for all primes $p$ and positive integers $\alpha$ \cite{schemmel69}. These interesting generalizations of the Euler totient function have applications in the study of magic squares \cite[page 184]{Sandor04} and in the enumeration of cliques in certain graphs \cite{Defant14B}. 
\par
Because $f$ is multiplicative, properties I and II of $f$ are equivalent to the following properties, which we will later reference.
\begin{enumerate}[A.] 
\item For all integers $n>1$, $f(n)<n$.
\item If $p$ is a prime divisor of a positive integer $n$, then $f(p)\vert f(n)$. 
\end{enumerate} 
\par 
Let $f^0(n)=n$ and $f^{k+1}(n)=f(f^k(n))$ for all nonnegative integers $k$ and $n$. 
Observe that, for any $n\in\mathbb{N}$, $f^n(n)\in\{0,1\}$. Furthermore, $f^n(n)=\displaystyle{\lim_{m\rightarrow\infty}f^m(n)}$, so we will define $\displaystyle{H(n)=\lim_{m\rightarrow\infty}f^m(n)}$. The author has shown that the function $H\colon\mathbb{N}\rightarrow\{0,1\}$ is completely multiplicative for the case in which $f$ is a Schemmel totient function \cite{Defant14A}. Our purpose is to prove that $H$ is completely multiplicative for any choice of a multiplicative arithmetic function $f$ that satisfies properties I, II, III, and IV. To help do so, we define the following sets.
\[P=\{p\in\mathbb{P}\colon H(p)=1\}\]
\[Q=\{q\in\mathbb{P}\colon H(q)=0\}\]
\[S=\{n\in\mathbb{N}\colon q\nmid n\:\forall\: q\in Q\}\]
We define $T$ to be the unique set of positive integers defined by the following criteria:
\begin{itemize}
\item $1\in T$.

\item If $p$ is prime, then $p\in T$ if and only if $f(p)\in T$.

\item If $x$ is composite, then $x\in T$ if and only if there exist $x_1,x_2\in T$ such that $x_1,x_2>1$ and $x_1x_2=x$.
\end{itemize} 
Note that $T$ is a set of \emph{positive} integers; in particular, $0\not\in T$. We may now establish a couple of lemmas that should make the proof of the desired theorem relatively painless. 
\begin{lemma} \label{Lem1.1}
Let $k\in\mathbb{N}$. If all the prime divisors of $k$ are in $T$, then all the positive divisors of $k$ (including $k$) are in $T$. Conversely, if $k\in T$, then every positive divisor of $k$ is an element of $T$.
\end{lemma}
\begin{proof}
First, suppose that all the prime divisors of $k$ are in $T$, and let $d$ be a positive divisor of $k$. Then all the prime divisors of $d$ are in $T$. Let $\displaystyle{d=\prod_{i=1}^rp_i^{\alpha_i}}$ be the canonical prime factorization of $d$. As $p_1\in T$, the third defining criterion of $T$ tells us that $p_1^2\in T$. Then, by the same token, $p_1^3\in T$. Eventually, we find that $p_1^{\alpha_1}\in T$. As $p_1^{\alpha_1},p_2\in T$, we have $p_1^{\alpha_1}p_2\in T$. Repeatedly using the third criterion, we can keep multiplying by primes until we find that $d\in T$. This completes the first part of the proof. Now we will prove that if $k\in T$, then every positive divisor of $k$ is an element of $T$. The proof is trivial if $k$ is prime, so suppose $k$ is composite. We will induct on $\Omega(k)$, the number of prime divisors (counting multiplicities) of $k$. If $\Omega(k)=2$, then, by the third defining criterion of $T$, the prime divisors of $k$ must be elements of $T$. Therefore, if $\Omega(k)=2$, we are done. Now, suppose the result holds whenever $\Omega(k)\leq h$, where $h>1$ is an integer. Consider the case in which $\Omega(k)=h+1$. By the third defining criterion of $T$, we can write $k=k_1k_2$, where $1<k_1,k_2<k$ and $k_1,k_2\in T$. By the induction hypothesis, all of the positive divisors of $k_1$ and all of the positive divisors of $k_2$ are in $T$. Therefore, all of the prime divisors of $k$ are in $T$. By the first part of the proof, we conclude that all of the positive divisors of $k$ are in $T$.
\end{proof}
\begin{lemma} \label{Lem1.2}
The sets $S$ and $T$ are equal. 
\end{lemma}
\begin{proof}
First, note that $1\in S\cap T$. Let $m>1$ be an integer such that,
for all $k\in \{1,2,\ldots,m-1\}$, either $k\in S\cap T$ or $k\not\in S\cup T$. We will show that $m\in S$ if and only
if $m\in T$. First, we must show that if $k\in \{1,2,\ldots,m-1\}$, then $k\in S$ if and only if $f(k)\in S$. Suppose, by way of contradiction, that $f(k)\in S$ and $k\not\in S$. As $k\not\in S$, we have that $k>1$ and $k\not\in T$. Lemma \ref{Lem1.1} then
guarantees that there exists a prime $q$ such that $q\vert k$ and $q\not\in T$. As $q\not\in T$, the second defining criterion of $T$ implies that $f(q)\not\in T$. As $f(k)\in S$, $f(k)\neq 0$. By property B of $f$, $f(q)\vert f(k)$, so $f(q)\neq 0$. Therefore, $f(q)\in \{1,2,\ldots,m-1\}$, and $f(q)\not\in T$. By the induction hypothesis, $f(q)\not\in S$. Therefore, there exists some $q_0\in Q$ such that $q_0\vert f(q)$. Thus, $q_0\vert f(q)\vert f(k)$, which contradicts the assumption that $f(k)\in S$. 
\par 
Conversely, suppose that $f(k)\not\in S$ and $k\in S$. The fact that $f(k)\not\in S$ implies that $k>1$, and the fact that $k\in S$ implies (by the induction hypothesis) that $k\in T$. By Lemma \ref{Lem1.1}, all prime divisors of $k$ are elements of $T$. The second criterion defining $T$ then implies that $f(p)\in T$ for all prime divisors $p$ of $k$. Using Lemma \ref{Lem1.1} again, we conclude that, for any prime divisor $p$ of $k$, all prime divisors of $f(p)$ are in $T$. By property III of $f$, all prime divisors of $f(k)$ are elements of $T$. Therefore, Lemma \ref{Lem1.1} guarantees that $f(k)\in T$. From property A of $f$ and the fact that $0\not\in T$, we see that $f(k)\in \{1,2,\ldots,m-1\}$. The induction hypothesis then implies that $f(k)\in S$, which is a contradiction. Thus, we have established that if $k\in \{1,2,\ldots,m-1\}$, then $k\in S$ if and only if $f(k)\in S$.
\par 
We are now ready to establish that $m\in S$ if and only if $m\in T$. Assume, first, that $m$ is prime. By the second criterion defining $T$, $m\in T$ if and only if $f(m)\in T$. By the induction hypothesis and property A of $f$, $f(m)\in T$ if and only if $f(m)\in S$. From the preceding argument, we see that $f(m)\in S$ if and only if $f^2(m)\in S$. Similarly, $f^2(m)\in S$ if and only if $f^3(m)\in S$. Continuing this pattern, we eventually find that $m\in T$ if and only if $f^m(m)\in S$. Observe that $f^m(m)=H(m)$ and that $0\not\in S$ and $1\in S$. Hence, $m\in T$ if and only if $H(m)=1$. Because $m$ is prime, $H(m)=1$ if and only if $m\not\in Q$. Finally, it follows from the definition of $S$ that $m\not\in Q$ if and only if $m\in S$. This completes the proof of the case in which $m$ is prime. 
\par 
Assume, now, that $m$ is composite. By Lemma \ref{Lem1.1}, $m\in T$ if and only if all the prime divisors of $m$ are in $T$. Because $m$ is composite, all the prime divisors of $m$ are elements of $\{1,2,\ldots,m-1\}$. Therefore, by the induction hypothesis, all the prime divisors of $m$ are in $T$ if and only if all the prime divisors of $m$ are in $S$. It should be clear from the definition of $S$ that all the prime divisors of $m$ are in $S$ if and only if $m\in S$. Hence, $m\in T$ if and only if $m\in S$. 
\end{proof} 
We may now use the sets $S$ and $T$ interchangeably. In addition, part of the above proof gives rise to the following corollary.
\begin{corollary} \label{Cor1.1}
Let $k,r\in \mathbb{N}$. Then $f^r(k)\in S$ if and only if $k\in S$.
\end{corollary}
\begin{proof}
The proof follows from the argument in the above proof that $f(k)\in S$ if and only if $k\in S$ whenever $k\in \{1,2,\ldots,m-1\}$. As we now know that we can make $m$ as large as we need, it follows that $f(k)\in S$ if and only if $k\in S$. Repeating this argument, we see that $f^2(k)\in S$ if and only if $f(k)\in S$. The proof then follows from repeated application of the same argument.  
\end{proof}
\begin{corollary} \label{Cor1.2}
For any positive integer $k$, $H(k)=1$ if and only if $k\in S$.
\end{corollary}
\begin{proof}
It is clear that $H(k)=1$ if and only if $H(k)\in S$. Therefore, the proof follows immediately from setting $r=k$ in Corollary \ref{Cor1.1}.
\end{proof}

Notice that Corollary \ref{Cor1.2}, Lemma \ref{Lem1.2}, and the defining criteria of $T$ provide a simple recursive means of constructing the set of positive integers $x$ that satisfy $H(x)=1$. We also have the following theorem. 
\begin{theorem} \label{Thm1.3}
The function $H\colon\mathbb{N}\rightarrow\{0,1\}$ is completely multiplicative. 
\end{theorem}
\begin{proof}
Corollary \ref{Cor1.2} tells us that $H$ is the characteristic function of the set $S$ of positive integers that are not divisible by primes in $Q$. If $x,y\in\mathbb{N}$, then it is clear that $xy\in S$ if and only if $x\in S$ and $y\in S$. The proof follows immediately. 
\end{proof}

\end{document}